\documentclass[12pt]{amsart} 

\usepackage[T1]{fontenc}  
\usepackage{lmodern}
\usepackage{microtype}
\usepackage{amsmath,amssymb,mathtools} 
\usepackage{silence}
\usepackage{enumitem}
\usepackage[margin=1in]{geometry}
\usepackage{xcolor}
\usepackage[colorlinks=true,linkcolor=blue,citecolor=blue,urlcolor=blue]{hyperref}
\hypersetup{
  pdftitle={Chevalley-Eilenberg cohomology of linearly reductive Lie algebras in the positive Verlinde category},
  pdfauthor={Pavel Etingof and Serina Hu}
}

\allowdisplaybreaks
\newtheorem{theorem}{Theorem}[section]
\newtheorem{proposition}[theorem]{Proposition}
\newtheorem{lemma}[theorem]{Lemma}
\newtheorem{corollary}[theorem]{Corollary}

\newtheorem{definition}[theorem]{Definition}

\theoremstyle{remark}
\newtheorem{remark}[theorem]{Remark}

\DeclareMathOperator{\Ver}{Ver}
\DeclareMathOperator{\Rep}{Rep}

\DeclareMathOperator{\Hom}{Hom}
\DeclareMathOperator{\Ext}{Ext}
\DeclareMathOperator{\gr}{gr}

\DeclareMathOperator{\Lie}{Lie}

\newcommand{\kk}{\Bbbk}
\newcommand{\cC}{\mathcal C}
\newcommand{\g}{\mathfrak g}
\newcommand{\one}{\mathbf 1}
\newcommand{\CE}{\mathrm{CE}}
\newcommand{\dR}{\mathrm{dR}}
\newcommand{\uHom}{\underline{\Hom}}

\definecolor{insertpurple}{RGB}{112,48,160}

\newcommand{\cK}{\mathcal K}
\newcommand{\cQ}{\mathcal Q}

\title[Lie algebra cohomology in $\Ver_p^+$]
{Chevalley-Eilenberg cohomology of linearly reductive Lie algebras in the Verlinde category}

\author{Pavel Etingof}
\address{Department of Mathematics, Massachusetts Institute of Technology, Cambridge, MA 02139, USA}
\email{etingof@math.mit.edu}

\author{Serina Hu}
\address{Department of Mathematics, Massachusetts Institute of Technology, Cambridge, MA 02139, USA}
\email{serinahu@mit.edu}

\date{}

\begin{document}

\begin{abstract}
Let $\kk$ be an algebraically closed field of characteristic $p\geq5$. Let
$\Ver_p$ be the Verlinde fusion category over $\kk$, i.e., the semisimplification of $\Rep_\kk(\mathbb Z/p)$, with simple objects 
$L_1=\one,L_2,\dots,L_{p-1}$. 
Let $\Ver_p^+$ be the even part of $\Ver_p$, with simple objects $L_1,L_3,\dots,L_{p-2}$. Let $\g$ be a linearly reductive Lie algebra in 
$\Ver_p^+$, i.e., one whose finite-dimensional representations are semisimple.

We first prove that $\g$ is invariantless, i.e., the unit object of $\Ver_p^+$ is not a 
summand of $\g$. Next, for
$3\leq m\leq p-2$, $m$ odd, set
$\g_m:=\Hom_{\Ver_p^+}(L_m,\g)$ and define the graded vector space
$ E_{\g}:=\bigoplus_{\substack{3\leq m\leq p-2\\m\text{ odd}}}
 \g_m^{(1)}[m]$,
where $(1)$ denotes Frobenius twist and $[m]$ the grading.
Our main result is that the Chevalley-Eilenberg cohomology $H^\bullet_{\CE}(\g)$ 
of $\g$ with trivial coefficients is isomorphic to $\wedge^\bullet E_{\g}^*$ as a graded algebra.
We also show that the algebra $H^\bullet_{\CE}(\g)$ is naturally isomorphic to the 
de Rham cohomology $H^\bullet_{\dR}(G)$ of the group scheme $G:=\exp(\g)$
and hence has a natural graded Hopf algebra structure given by multiplication in $G$, and prove that it corresponds to the usual graded Hopf algebra structure on $\wedge^\bullet E_{\g}^*$. 

For every finite $\g$-module $V$, we prove
$H^\bullet_{\CE}(\g,V)\cong\wedge^\bullet E_{\g}^*\otimes V^{\g}$
as a module over $H^\bullet_{\CE}(\g)$.
We also prove a relative version: if
$\mathfrak k\subset\g$ is a Lie subalgebra and $E_{\mathfrak k}$ is
constructed from $\mathfrak k$ in the same way as $E_\g$, then
$H^\bullet_{\CE}(\g,\mathfrak k)\cong
\wedge^\bullet(E_\g/E_{\mathfrak k})^*$, and, more generally,
$H^\bullet_{\CE}(\g,\mathfrak k;V)\cong
\wedge^\bullet(E_\g/E_{\mathfrak k})^*\otimes V^\g$.
This implies the classical theorem of Borel and Chevalley on the cohomology of complex semisimple Lie algebras, as well as its analogue in sufficiently large positive characteristic.

The proof uses the vanishing of Exts in the semisimple category
$\Rep(\g)$ together with the calculation of the Koszul homology of the
simple objects of $\Ver_p^+$.
\end{abstract}

\maketitle

\tableofcontents

\section{Introduction}

For a complex semisimple Lie algebra $\g_{\mathbb C}={\rm Lie}(G_{\Bbb C})$ with exponents
$m_1,\dots,m_r$, a classical theorem of Borel and Chevalley gives an isomorphism of graded Hopf algebras
\[
 H^\bullet(\g_{\mathbb C},\mathbb C)
 \cong \wedge^\bullet(\xi_1,\dots,\xi_r),
 \qquad |\xi_i|=2m_i+1,
\]
where the coproduct on $H^\bullet(\g_{\mathbb C},\mathbb C)$ is 
defined, upon its identification with the de Rham cohomology 
$H^\bullet_{\rm dR}(G_{\Bbb C},\Bbb C)$, by the multiplication in $G_{\Bbb C}$; see \cite{Ree95} for a modern proof. Also for every nontrivial finite-dimensional irreducible $\g_{\mathbb C}$-module $V$ one has $H^\bullet(\g_{\mathbb C},V)=0$.
The purpose of this paper is to establish the corresponding statements for
linearly reductive Lie algebras in the Verlinde category $\Ver_p^+$, i.e., those whose finite-dimensional representations are completely reducible. Such 
Lie algebras are discussed in \cite[Section~4]{CEO24}; for example, they arise as images under semisimplification of usual simple Lie algebras $\g_\kk$ over a field $\kk$ of characteristic $p$ with action of $\mathbb Z/p$ by a principal unipotent element, when $p$ exceeds the Coxeter number of $\g_\kk$. 

Namely, our results are as follows. 
Let $\kk$ be an algebraically closed field of characteristic $p\geq5$. Let
$\Ver_p$ be the Verlinde fusion category over $\kk$, i.e., the semisimplification of $\Rep_\kk(\mathbb Z/p)$, with simple objects 
$L_1=\one,L_2,\dots,L_{p-1}$. 
Let $\Ver_p^+$ be the even part of $\Ver_p$, with simple objects $L_1,L_3,\dots,L_{p-2}$. Let $\g$ be a linearly reductive 
Lie algebra in $\Ver_p^+$. 

We first show that $\g$ is {\it invariantless}, i.e., the classical part
$\g_1:=\Hom_{\Ver_p^+}(\one,\g)$ of $\g$ vanishes (Proposition~\ref{prop:LR-invariantless}). The proof uses the relative PBW decomposition which
makes $U(\g)$ a finite free right module over $U(\g_1)$. Shapiro's lemma then
transfers the vanishing of Ext in $\Rep(\g)$ to $\Rep(\g_1)$.  Since the category of ordinary finite-dimensional modules over
a nonzero Lie algebra in characteristic $p$ is never semisimple, one obtains
$\g_1=0$. It follows that $U(\g)$ is a finite algebra object in $\Ver_p^+$.

We then proceed to study the (co)homology of $\g$. 
For a $\g$-module $V$, let 
$C^\bullet_{\CE}(\g,V)=\wedge^\bullet \g^*\otimes V$ be the 
standard complex of $\g$ with coefficients in $V$ with the usual Chevalley-Eilenberg differential, also called the {\it Chevalley-Eilenberg complex}. 
If $V=\one$, we denote it just by $C^\bullet_{\CE}(\g)$. 
The cohomology of this complex is called the {\it Chevalley-Eilenberg cohomology} of $\g$ with coefficients in $V$ and denoted $H^\bullet_{\CE}(\g,V)$. 
In particular, we denote $H^\bullet_\CE(\g,\one)$ just by $H^\bullet_\CE(\g)$. 
We also write
\[
 C_\bullet^{\CE}(\g):=\wedge^\bullet\g,
 \qquad H_\bullet^{\CE}(\g):=H_\bullet(C_\bullet^{\CE}(\g))
\]
for the Chevalley-Eilenberg chain complex and homology of $\g$ with trivial
coefficients.

Equivalently, $C^n_{\CE}(\g,V)=\uHom(\wedge^n\g,V)$, where
$\uHom$ denotes the internal Hom in $\Ver_p^+$; we use $\Hom$ for the
ordinary vector space of morphisms.  Thus cohomology with coefficients is
an object of $\Ver_p^+$.  The differentials are recalled in
Subsection~\ref{subsec:CE-complexes}.

Proposition~\ref{prop:dR-CE} shows, as in the complex case, that
$H^\bullet_{\CE}(\g)\cong H^\bullet_{\rm dR}(G)$, where $G=\exp(\g)$
is the height-one linearly reductive affine group scheme corresponding to $\g$.
This endows $H^\bullet_{\CE}(\g)$ with a natural Hopf algebra structure, whose
coproduct comes from multiplication in $G$.
 
Next, we prove our main result. Let $M^{(1)}$ denote the Frobenius twist of a finite-dimensional vector space $M$: scalar multiplication by $\lambda$ on $M$ induces multiplication by $\lambda^p$ on $M^{(1)}$.

\begin{theorem}[Main theorem]\label{thm:main-intro}
Let
\[
 \g=\bigoplus_{\substack{3\leq m\leq p-2\\m\text{ odd}}}
 \g_m\otimes L_m
\]
be a linearly reductive Lie algebra in $\Ver_p^+$, where
$\g_m:=\Hom_{\Ver_p^+}(L_m,\g)$, and set
\[
 E_{\g}:=
 \bigoplus_{\substack{3\leq m\leq p-2\\m\text{ odd}}}
 \g_m^{(1)}[m].
\]
Then:
\begin{enumerate}[label=\textup{(\roman*)},leftmargin=2.3em]
\item there is a canonical isomorphism of graded Hopf algebras\footnote{Here and below we identify a finite-dimensional $\kk$-vector space $Y$ 
 with the object $Y\otimes \one$ of a tensor category. Also we agree that the (co)homological degree is unchanged under taking duals.}
\[
 H^\bullet_{\CE}(\g)\cong\wedge^\bullet E_{\g}^*.
\]
Dually, there is a canonical isomorphism of graded Hopf algebras
\[
 H_\bullet^{\CE}(\g)\cong\wedge^\bullet E_\g.
\]
Moreover, the natural maps from invariant cochains and invariant chains induce
isomorphisms
\[
 (\wedge^\bullet\g^*)^\g\xrightarrow{\ \sim\ }H^\bullet_{\CE}(\g),
 \qquad
 (\wedge^\bullet\g)^\g\xrightarrow{\ \sim\ }H_\bullet^{\CE}(\g).
\]
Thus the cohomology and homology isomorphisms above are dual under the natural
pairing.

\item If $V$ is a simple $\g$-module with nonzero $\g$-action, then
\[
 H^\bullet_{\CE}(\g,V)=0.
\]
\end{enumerate}
Thus for every finite $\g$-module $V$ there is a natural
isomorphism
\[
 H^\bullet_{\CE}(\g,V)
 \cong \wedge^\bullet E_{\g}^*\otimes V^{\g}
\]
of graded modules over $H^\bullet_{\CE}(\g)\cong \wedge^\bullet E_\g^*$.  In particular,
\[
 \sum_{n\geq0}\dim_{\kk}H^n_{\CE}(\g)t^n
 =\prod_{\substack{3\leq m\leq p-2\\m\text{ odd}}}
 (1+t^m)^{\dim\g_m}.
\]
\end{theorem}

There is also a relative version. If $\mathfrak k\subset\g$ is a Lie
subalgebra, then $\mathfrak k$ is linearly reductive by
Lemma~\ref{subal}. The inclusion induces
$E_{\mathfrak k}\hookrightarrow E_\g$; setting
$E_{\g/\mathfrak k}:=E_\g/E_{\mathfrak k}$, we prove in
Theorem~\ref{thm:relative-main} that
\[
 H^\bullet_{\CE}(\g,\mathfrak k)
 \cong\wedge^\bullet E_{\g/\mathfrak k}^*,
 \qquad
 H^\bullet_{\CE}(\g,\mathfrak k;V)
 \cong\wedge^\bullet E_{\g/\mathfrak k}^*\otimes V^\g.
\]
Thus, if $D_{\mathfrak a}$ denotes the multiset of degrees of the
primitive generators of $H^\bullet_{\CE}(\mathfrak a)$, then
$D_{\mathfrak k}$ is a submultiset of $D_\g$ and relative cohomology is
the exterior algebra on the multiset difference
$D_\g\setminus D_{\mathfrak k}$.

The idea of the proof of the main theorem is as follows. Since $U(\g)$ is finite and finite-dimensional $U(\g)$-modules are completely reducible, it follows that $U(\g)$ is a semisimple algebra, so 
the higher Ext groups between its modules must vanish. Nevertheless, the Chevalley-Eilenberg complex $C_{\CE}^\bullet(\g)=\wedge^\bullet \g^*$ 
need not be exact, as, in contrast with the classical case, it does not compute these Ext groups. 
Namely, $C_{\CE}^\bullet(\g)=\uHom_{U(\g)}(U(\g)\otimes \wedge^\bullet \g,\one)$, 
where $U(\g)\otimes \wedge^\bullet \g$ is the quantized Koszul complex of $\g$ 
whose PBW associated graded is the ordinary Koszul complex of the underlying object $\g$ (\cite{Eti18}). So the key reason for non-exactness of $C_{\CE}^\bullet(\g)$ 
is that this Koszul complex is not exact, i.e., the object $\g\in \Ver_p^+$ is not Koszul (\cite{Eti18}). Indeed, for $2\leq m\leq p-2$, the Koszul complex of the simple object $L_m$ has a nonzero homology class in homological degree $m$ and
diagonal degree $p$ \cite[Proposition~5.1]{Eti18}.  These classes account
exactly for the discrepancy between the Ext cohomology (which is zero in positive degrees) and
the Chevalley-Eilenberg cohomology, which leads to the proof of the main result.
More precisely, we show that the corresponding spectral sequence degenerates by using its multiplicative structure and the bidegrees of the almost-Koszul generators.
The resulting discrepancy between derived invariants and Chevalley-Eilenberg
cohomology is summarized in Remark~\ref{rem:derived-invariants}.

The organization of the paper is as follows. Section 2 contains preliminaries, a proof that a linearly reductive Lie algebra in $\Ver_p^+$ is invariantless, and an interpretation of the Chevalley-Eilenberg cohomology of a linearly reductive Lie algebra as de Rham cohomology of the corresponding group scheme.
In Section 3 we prove the main theorem. In Section 4 we give examples. In Section 5 we discuss applications to ordinary Lie algebra cohomology. Finally, Section~6 proves the relative version of the main theorem, including cohomology with coefficients. 

{\bf Acknowledgements.} In this paper we collaborated with ChatGPT 5.5 and 5.6 Pro. This work was partially supported by the NSF grant DMS-2502467. 

\section{Preliminaries}

\subsection{The categories $\Ver_p$ and \texorpdfstring{$\Ver_p^+$}{Ver-p-plus} and finiteness of symmetric and exterior algebras}

Recall that the Verlinde category $\Ver_p$ is the semisimplification of
$\Rep_{\kk}(\mathbb Z/p\mathbb Z)$.  Its simple objects are
$L_1=\one,L_2,\dots,L_{p-1}$, all of which are self-dual.  The invertible object $L_{p-1}$ generates a copy
of $\mathrm{sVec}$, and $\Ver_p=\Ver_p^+\boxtimes\mathrm{sVec}$, 
where $\Ver_p^+$ is generated by the $L_m$ with $m$ odd; see e.g. 
\cite{EOV17,Eti18}.

For $X\in\Ver_p$, set
$X_m:=\Hom_{\Ver_p}(L_m,X)$. 
The evaluation maps give a canonical decomposition
$
 X\cong\bigoplus_{1\leq m\leq p-1}X_m\otimes L_m.
$
Thus each $X_m$ is an ordinary finite-dimensional vector space.  We call $X$
\emph{invariantless} if $X_1=0$.

\begin{lemma}\label{lem:finiteness}
If $X\in\Ver_p^+$ is invariantless, then $SX$ and $\wedge X$ are finite
objects.  
\end{lemma}

\begin{proof}
For $2\leq m\leq p-2$, one has
$S^iL_m=0\ (i>p-m),\ 
 \wedge^jL_m=0\ (j>m)$;
see \cite[Proposition~2.4]{EOV17} and \cite[Proposition~5.1]{Eti18}.
Both $S$ and $\wedge$ take direct sums to graded tensor products.  After
choosing bases in the multiplicity spaces, $X$ is a finite direct sum of copies
of $L_m$ with $2\leq m\leq p-2$, proving the asserted finiteness.
\end{proof}

\subsection{Exterior powers}\label{subsec:exterior-conventions}

Let $\cC$ be a symmetric tensor category over $\kk$; see
\cite[Definitions~4.1.1, 8.1.1, and 8.1.12]{EGNO15}.  For $V\in\cC$, let
$\wedge^nV$ be the ordinary exterior power and let
$\Lambda^nV:=\left(\wedge^nV^*\right)^*$
be the dual exterior power.  Thus $\wedge^nV$ is the object of
anticoinvariants in $V^{\otimes n}$, while $\Lambda^nV\subset V^{\otimes n}$
is the object of antiinvariants.  Set
\[
 \wedge V:=\bigoplus_{n\geq0}\wedge^nV,
 \qquad
 \Lambda V:=\bigoplus_{n\geq0}\Lambda^nV.
\]
Both are graded Hopf algebras.

\begin{lemma}\label{lem:exterior-comparison}
(i) For every $V\in\cC$ and $n\geq0$, the alternating morphism
\[
 e_n^-:=\sum_{\sigma\in S_n}\operatorname{sgn}(\sigma)\sigma:
 V^{\otimes n}\longrightarrow\Lambda^nV
\]
descends to a morphism
\[
 \xi_V^n:\wedge^nV\longrightarrow\Lambda^nV.
\]
The morphisms $\xi_V^n$ define a natural morphism of graded Hopf algebras
\[
 \xi_V:\wedge V\longrightarrow\Lambda V.
\]
Moreover, under the canonical identifications
\[
 \wedge(V\oplus W)\cong\wedge V\otimes\wedge W,
 \qquad
 \Lambda(V\oplus W)\cong\Lambda V\otimes\Lambda W,
\]
one has $\xi_{V\oplus W}=\xi_V\otimes\xi_W$.  

(ii) If $V\in\Ver_p^+$, then
$\xi_V$ is an isomorphism.
\end{lemma}

\begin{proof}
(i) Let $s_i$ denote the symmetry interchanging the $i$-th and $(i+1)$-st
factors of $V^{\otimes n}$.  One has
\[
 e_n^-(1+s_i)=
 (1+s_i)e_n^-=0.
\]
The first identity shows that $e_n^-$ annihilates the relations defining
$\wedge^nV$, and the second shows that its image is contained in the
antiinvariants $\Lambda^nV$.  Hence $e_n^-$ induces $\xi_V^n$.

The multiplication in $\Lambda V$ is dual to the comultiplication in
$\wedge V^*$, and the $n$-fold product $(\Lambda^1V)^{\otimes n}=V^{\otimes n}\to \Lambda^n V$ is the
alternating morphism $e_n^-$.  Thus $\xi_V$ is the algebra map induced by the
identity of $V$.  Since $V$ consists of primitive elements in both Hopf
algebras, this map also preserves comultiplication. Finally, under the direct-sum identifications, both
$\xi_{V\oplus W}$ and $\xi_V\otimes\xi_W$ are Hopf algebra maps extending the
identity on $V\oplus W$, so they coincide.

For $n<p$, let
\[
 \psi_V^n:\Lambda^nV\longrightarrow\wedge^nV
\]
be the composite of the inclusion into $V^{\otimes n}$ with the quotient onto
$\wedge^nV$.  Then
\[
 \psi_V^n\circ \xi_V^n=n!\,\mathrm{id}_{\wedge^nV},
 \qquad
 \xi_V^n\circ \psi_V^n=n!\,\mathrm{id}_{\Lambda^nV},
\]
so $\xi_V^n$ is an isomorphism for $n<p$.

(ii) Now let $V=L_m$ be simple in $\Ver_p^+$.  If $m>1$, then
$\Lambda^nL_m=0$ for $n>m$ by \cite[Proposition~5.1]{Eti18}; since $L_m$ is
self-dual, the same holds for $\wedge^nL_m$.  As $m\leq p-2$, all nonzero degrees are
strictly smaller than $p$, and hence $\xi_{L_m}$ is an isomorphism.  The same
is clear for $L_1=\one$.  But every object of $\Ver_p^+$ is a finite direct sum of
such simple objects, so the direct-sum compatibility proves the result for
all $V\in\Ver_p^+$.
\end{proof}

Henceforth all exterior powers are taken in $\Ver_p^+$, 
so we identify $\wedge^nV$ with $\Lambda^nV$ via $\xi_V^n$ and
denote both exterior algebras by $\wedge V$.

\begin{remark}\label{rem:exterior-failure} In general, $\xi_V$ need not be injective or surjective.  For example, if
$V$ is the one-dimensional odd vector space in $\mathrm{sVec}$, then
$\wedge^nV=\Lambda^nV=V^{\otimes n}$ and $\xi_V^n=n!\,\mathrm{id}$; hence
$\xi_V^p=0$ although both its source and target are nonzero. 
 \end{remark}

\subsection{Chevalley-Eilenberg complexes}\label{subsec:CE-complexes}

Below we write tensor-categorical morphism identities in schematic element
notation, which acquire literal meaning upon applying the Grothendieck functor of points 
over a commutative ind-algebra. In particular, 
permutations of arguments are interpreted using the symmetry,
and formulas like~\eqref{eq:CE-cochain-differential} and
\eqref{eq:CE-chain-differential} below denote the corresponding morphisms between
tensor products.

For a Lie algebra $\g$ in $\Ver_p^+$ and a finite $\g$-module $V$, the
differential on $C^n_{\CE}(\g,V)=\uHom(\wedge^n\g,V)$ is given by
\begin{equation}\label{eq:CE-cochain-differential}
\begin{aligned}
(d_{\CE}f)(x_0,\ldots,x_n)
 &=\sum_{i=0}^n(-1)^i x_i\cdot
 f(x_0,\ldots,\widehat{x_i},\ldots,x_n)\\
 &\quad+\sum_{0\leq i<j\leq n}(-1)^{i+j}
 f([x_i,x_j],x_0,\ldots,\widehat{x_i},\ldots,
 \widehat{x_j},\ldots,x_n),
\end{aligned}
\end{equation}
where the arguments not carrying hats retain their order; for trivial
coefficients the first sum is absent.  Dually, the boundary map on chains is
\begin{equation}\label{eq:CE-chain-differential}
 \partial_{\CE}(x_1\wedge\cdots\wedge x_n)
 =\sum_{1\leq i<j\leq n}(-1)^{i+j}[x_i,x_j]\wedge
 x_1\wedge\cdots\wedge\widehat{x_i}\wedge\cdots\wedge
 \widehat{x_j}\wedge\cdots\wedge x_n.
\end{equation}

\subsection{Linear reductivity and invariantlessness}
\label{subsec:classical-part}

By an {\it operadic Lie algebra} in a symmetric tensor category we mean an object with a commutator satisfying skew-symmetry and the Jacobi identity. By \cite[Theorem~6.6]{Eti18}, every operadic Lie algebra 
in $\Ver_p^+$ is a Lie algebra in the sense of \cite[Definition~4.6]{Eti18} (as $\Hom_{\Ver_p}(L_2,\g)=0$) and satisfies PBW. Since the map $S\g\to {\rm gr}U(\g)$ is surjective (in fact, an isomorphism), 
 if $\g$ is an invariantless Lie algebra in
$\Ver_p^+$ then $U(\g)$ is finite by Lemma \ref{lem:finiteness}.

\begin{definition}\label{def:linear-reductive}
A Lie algebra $\g$ in $\Ver_p^+$ is \emph{linearly reductive} if the category
$\Rep(\g)$ of finite $U(\g)$-modules is semisimple. 
\end{definition}

 For a $\g$-module $V$,
let $V^{\g}:=\uHom_{U(\g)}(\one,V)$. 

\begin{proposition}\label{prop:LR-invariantless}
Every linearly reductive Lie algebra in $\Ver_p^+$ is invariantless.
\end{proposition}

\begin{proof} We will use 

\begin{lemma}\label{lem:ordinary-not-semisimple} (\cite[Theorem~8, p.~30]{Zas54})
The category of finite-dimensional representations of a non-zero finite-dimensional Lie algebra in positive characteristic is not semisimple.
\end{lemma}

Let $\g_1:=\Hom_{\Ver_p^+}(\one,\g)$; it is an ordinary Lie algebra over $\kk$.
Then $\g=\g_1\oplus\g_{\ne 1}$ and $\Hom_{\Ver_p^+}(\one,\g_{\ne 1})=0$.
Setting $\deg \g_1=0$ and $\deg \g_{\ne 1}=1$ gives a finite filtration of
$U(\g)$ as a right $U(\g_1)$-module with associated graded
$S(\g_{\ne 1})\otimes U(\g_1)$. 
Its successive quotients are free, and $S(\g_{\ne 1})$ is finite by
Lemma~\ref{lem:finiteness}; hence $U(\g)$ is a finite free right
$U(\g_1)$-module.

It follows that induction from $\g_1$ to $\g$ is exact and for every finite-dimensional ordinary $\g_1$-module $X$, the $\g$-module $U(\g)\otimes_{U(\g_1)}X$ is finite.  Shapiro's lemma for external Yoneda
Ext therefore gives
\[
 \Ext^1_{\g_1}(X,\kk)
 \cong
 \Ext^1_{\Rep(\g)}
 \bigl(U(\g)\otimes_{U(\g_1)}X,\one\bigr)=0.
\]
The Ext group on the left is the ordinary one: an extension of the classical
objects $X$ and $\one$ in $\Ver_p^+$ has classical middle term, since
$\Ver_p^+$ is semisimple.  Consequently, for finite-dimensional
$\g_1$-modules $X,Y$,
\[
 \Ext^1_{\g_1}(X,Y)
 \cong \Ext^1_{\g_1}(Y^*\otimes X,\kk)=0.
\]
Thus the category of finite-dimensional $\g_1$-modules is semisimple, so
Lemma~\ref{lem:ordinary-not-semisimple} implies $\g_1=0$.
\end{proof}

\begin{lemma}\label{subal} A Lie subalgebra of a linearly reductive Lie algebra in $\Ver_p^+$ is linearly reductive.
\end{lemma} 

\begin{proof} Let $\g$ be a linearly reductive Lie algebra in $\Ver_p^+$ and 
$\mathfrak{a}\subset \g$ a Lie subalgebra. By Proposition \ref{prop:LR-invariantless},
$\g$ is invariantless, hence so is $\mathfrak{a}$. Thus by the PBW theorem 
$U(\mathfrak{a})\to U(\g)$ is an inclusion of finite dimensional Hopf algebras. 
It follows that the restriction functor $\Rep(\g)\to \Rep(\mathfrak{a})$ 
is a surjective tensor functor. Thus the Lemma follows from 
Example 3.3(ii),(iii) of \cite{EO04}.
\end{proof} 

\subsection{De Rham cohomology and the Hopf structure on
\texorpdfstring{$H^\bullet_{\CE}(\g)$}{CE cohomology}}
\label{sec:de-rham}

Let $\cC=\Ver_p^+$ and work in $\operatorname{Ind}(\cC)$.  For a finitely
generated commutative algebra object $A$, put $X=\operatorname{Spec}A$ and
let $\Omega_A^1$ be its module of K\"ahler differentials.  The de Rham
complex and de Rham cohomology of $X$ are
\[
 \Omega^\bullet(X):=
 \bigl(\Omega_A^0\xrightarrow{d}\Omega_A^1\xrightarrow{d}
 \Omega_A^2\xrightarrow{d}\cdots\bigr),
 \qquad
 H^\bullet_{\dR}(X):=H^\bullet(\Omega^\bullet(X)),
\]
where $\Omega_A^i:=\wedge_A^i\Omega_A^1$ (so $\Omega_A^0=A$). 
For affine schemes of finite type one has naturally
\[
 \Omega^\bullet(X\times Y)
 \cong \Omega^\bullet(X)\otimes\Omega^\bullet(Y).
\]
Since tensor product is exact in $\operatorname{Ind}(\cC)$, this gives the
K\"unneth isomorphism
\[
 H^\bullet_{\dR}(X\times Y)
 \cong H^\bullet_{\dR}(X)\otimes H^\bullet_{\dR}(Y).
\]
Also if $X$ is an affine space (i.e., an object of $\Ver_p^+$) then $\Omega^\bullet(X)$ coincides with the de Rham complex defined in \cite{Eti18}, Subsection 2.4 (as in $\Ver_p^+$, $\wedge X\cong \Lambda X$).

If $G$ is an affine group scheme of finite type in $\cC$, pullback by
multiplication, the unit, and inversion makes $\Omega^\bullet(G)$ a
commutative differential graded Hopf algebra.  In particular,
\[
 m^*:H^\bullet_{\dR}(G)
 \longrightarrow H^\bullet_{\dR}(G\times G)
 \cong H^\bullet_{\dR}(G)\otimes H^\bullet_{\dR}(G)
\]
is a coproduct, and $H^\bullet_{\dR}(G)$ is a graded Hopf algebra.

Let $\g=\Lie(G)$.  Left translation gives a canonical trivialization
\begin{equation}\label{eq:left-trivialization-forms}
 \Omega^\bullet(G)
 \cong \mathcal O(G)\otimes\wedge^\bullet\g^*.
\end{equation}
Under this identification the de Rham differential is the
Chevalley-Eilenberg differential with coefficients in the regular
$\g$-module $\mathcal O(G)$, while the subcomplex of left-invariant forms is
$C^\bullet_{\CE}(\g)$.

A finite group scheme $G$ is called \emph{linearly reductive} if its
category of representations is semisimple. We
call $G$ \emph{height one} if
its distribution algebra $\mathcal O(G)^*$ is generated, as an algebra, by
$\g=\Lie(G)$.  This is the property of the height-one group schemes used in
\cite[Section~4]{CEO24}; in particular, the group scheme 
$G=\exp(\g), \mathcal O(G)^*=U(\g)$
has height one.

\begin{proposition}\label{prop:dR-CE}
Let $G$ be a finite linearly reductive height-one group scheme in $\Ver_p^+$,
and let $\g=\Lie(G)$.  Inclusion of left-invariant forms induces a canonical
isomorphism of graded algebras
\[
 H^\bullet_{\CE}(\g)\xrightarrow{\ \sim\ }H^\bullet_{\dR}(G).
\]
Consequently, $H^\bullet_{\CE}(\g)$ has a canonical graded Hopf algebra
structure.  
\end{proposition}

\begin{proof}
This is the usual averaging argument for invariant differential forms; cf.\
\cite{CE48,HK62}.  Let
$\operatorname{Av}_M:M\to M$ be the Reynolds idempotent for a 
$G$-module $M$, whose image is $M^G$.  Equivalently, if
$\lambda:\mathcal O(G)\to\one$ is the normalized invariant integral and
$\rho_M:M\to\mathcal O(G)\otimes M$ is the coaction, then
\[
 \operatorname{Av}_M=(\lambda\otimes\operatorname{id}_M)\rho_M.
\]
Linear reductivity is precisely what makes this averaging operation exact.
Applied degreewise to $\Omega^\bullet(G)$, it is a chain projection onto the
subcomplex of left-invariant forms, and hence
\[
 H^\bullet(\Omega^\bullet(G)^G)
 \cong H^\bullet_{\dR}(G)^G.
\]

The schematic Cartan formula
$\mathcal L_x=d\,\iota_x+\iota_xd$, $x\in\g$,
shows that $\g$ acts trivially on $H^\bullet_{\dR}(G)$.  Since $G$ has height
one, the algebra $\mathcal O(G)^*$ is generated by $\g$, so the whole
$G$-action on de Rham cohomology is trivial.  Thus averaging acts as the
identity on cohomology.  Concretely, if $\omega$ is closed, then
$\operatorname{Av}(\omega)$ is a left-invariant closed form representing
$[\omega]$; and if a left-invariant form $\omega$ is $d\beta$, then
$\omega=d\operatorname{Av}(\beta)$.  Hence inclusion of left-invariant
forms is a quasi-isomorphism.  By
\eqref{eq:left-trivialization-forms}, its source is
$C^\bullet_{\CE}(\g)$, as desired.
\end{proof}

\section{Proof of the main theorem}\label{sec:proof-main}

By Proposition~\ref{prop:LR-invariantless}, $\g$ is invariantless; hence
Lemma~\ref{lem:finiteness} implies that $U(\g)$ and all complexes below are finite.

\subsection{The Koszul complex}

For an invariantless object
\[
 X=\bigoplus_{\substack{3\leq m\leq p-2\\m\text{ odd}}}X_m\otimes L_m
 \in\Ver_p^+,
\]
set
\[
\cK_n(X):=S X\otimes\wedge^nX.
\]
It is graded by the homological degree $n$ and by the diagonal degree, in
which both copies of $X$ have degree one, and we have the Koszul differential 
$d: \cK_n(X)\to \cK_{n-1}(X)$, giving $\cK_\bullet(X)$ the structure of a chain complex. This complex is called the {\it Koszul complex} of $X$ (\cite{Eti18}), and it has a natural structure of a graded commutative and cocommutative differential 
Hopf algebra.

Let $E_X$ be the graded vector
space $\bigoplus_m X_m^{(1)}$, with $X_m^{(1)}$ placed in degree $m$; for
$X=\g$ this is the grading denoted by
$E_\g=\bigoplus_m\g_m^{(1)}[m]$ in Theorem~\ref{thm:main-intro}.

\begin{proposition}\label{prop:koszul-general}
There is a canonical isomorphism of bigraded Hopf algebras
\[
 H_\bullet(\cK_\bullet(X))\cong\wedge^{\bullet,\bullet} E_X.
\]
The generators in $X_m^{(1)}$ have homological degree $m$, diagonal degree
$p$, and are primitive.
\end{proposition}

\begin{proof}
For $m$ odd, $3\leq m\leq p-2$, the calculation of the homology of the Koszul
complex in \cite{Eti18} gives
$H_0=\one, H_m=L_{p-1}^{\otimes(m+1)}=\one$
with no other homology, and the degree-$m$ class has diagonal degree $p$
(\cite[Proposition~5.1]{Eti18}).  This implies the statement for 
$X=L_m$.

Now consider the case of general $X$. Both $S$ and $\wedge$ take direct sums to graded tensor products, so the
K\"unneth formula gives an exterior generator for every copy of $L_m$ in
$X$.  Functoriality shows that a linear map with matrix $(a_{ij})$ acts on
these diagonal-degree-$p$ classes by $(a_{ij}^p)$; hence their multiplicity
space is canonically $X_m^{(1)}$ (cf. \cite[Corollary~5.2]{Eti18}). Finally, the reduced coproduct of such a
class vanishes, since there is no positive homology in diagonal degrees
strictly between $0$ and $p$.  Thus the generators are primitive.
\end{proof}

\subsection{The quantized Koszul complex}\label{sec:quantized-koszul}

We use homological grading: $\g[1]$ denotes $\g$ placed in degree one.
Set $\cQ_\bullet(\g):=U(\g)\ltimes S^\bullet (\g[1])=U(\g)\ltimes \wedge^\bullet \g$, where the smash product is formed using the adjoint action, and define the differential 
on $\cQ_\bullet(\g)$ by the schematic formulas
$d(x)=0,\ d(x[1])=x$. 
Equivalently, $\cQ_n(\g)=U(\g)\otimes\wedge^n\g$,
with the standard Chevalley-Eilenberg chain differential.

\begin{lemma}\label{lem:quantized-koszul}
The complex $\cQ_\bullet(\g)$ is a differential graded cocommutative Hopf
algebra.  For the filtration
\begin{equation}\label{disfil}
 F_s\cQ_n(\g):=
 F^{\mathrm{PBW}}_{s-n}U(\g)\otimes\wedge^n\g,
 \qquad F^{\mathrm{PBW}}_aU(\g)=0\quad(a<0),
\end{equation}
one has
\[
 \gr_F\cQ_\bullet(\g)\cong
 S\g\otimes\wedge^\bullet\g=\cK_\bullet(\g)
\]
as differential graded Hopf algebras.
\end{lemma}

\begin{proof}
The smash product is generated by primitive elements $x$ and $x[1]$, with schematic
relations
\[
 x\,y[1]-y[1]x=[x,y][1],
 \qquad x[1]y[1]+y[1]x[1]=0.
\]
The assignments $d(x)=0$ and $d(x[1])=x$ preserve these relations and define
a derivation and coderivation; expanding it gives the usual
Chevalley-Eilenberg differential.  In filtration \eqref{disfil}, the action
part of the differential preserves the filtration degree and the bracket part lowers it
by one.  PBW therefore identifies the associated graded differential with
the ordinary Koszul differential.
\end{proof}

We call the complex $\cQ_\bullet(\g)$ the {\it quantized Koszul complex} of $\g$. For every finite $\g$-module $V$, 
$\uHom_{U(\g)}(\cQ_\bullet(\g),V)
 \cong\wedge^\bullet\g^*\otimes V$
is the Chevalley-Eilenberg cochain complex of $\g$. 

\subsection{Homology of the quantized Koszul complex}

\begin{proposition}\label{prop:degeneration}
The spectral sequence of the filtration in
Lemma~\ref{lem:quantized-koszul} degenerates at $E_2$ (where $E_1$ is the associated graded complex and $E_2$ is its homology), and
\[
 \gr_F H_\bullet(\cQ_\bullet(\g))\cong\wedge^{\bullet,\bullet} E_\g
\]
as bigraded Hopf algebras.
\end{proposition}

\begin{proof}
We have
\[
 E_2^{s,n}=H_n(\cK_\bullet(\g))_s
 \Longrightarrow H_n(\cQ_\bullet(\g)),
 \qquad
 d_r:E_r^{s,n}\longrightarrow E_r^{s-r+1,n-1}.
\]
By Proposition~\ref{prop:koszul-general}, $E_2$ is generated by classes of
bidegree $(s,n)=(p,m)$, with $m\geq3$ odd.  Inductively, $d_r$ is a
derivation on the exterior algebra generated by the surviving classes.  On a
generator its target has diagonal degree $p-r+1$: this degree is absent for
$2\leq r\le p$, has positive homological degree and diagonal degree zero for
$r=p+1$, and is negative for $r>p+1$.  Hence every $d_r$ vanishes, so
$E_2=E_\infty$.
\end{proof}

The following filtered Hopf algebra lemma removes the extension ambiguity.

\begin{lemma}\label{lem:hopf-rigidity}
Let $H$ be a connected, nonnegatively graded cocommutative Hopf algebra in a
semisimple symmetric tensor category over a field of characteristic different from two,
with a finite Hopf filtration.  If
$\gr_FH\cong\wedge^\bullet E$
where $E$ is an ordinary vector space concentrated in odd homological degrees
and filtration degree $p$, then the generators of $E$ have unique primitive
lifts and these induce an isomorphism $H\cong\wedge^\bullet E$
of filtered graded Hopf algebras.
\end{lemma}

\begin{proof}
Since $\gr_FH$ is a direct sum of copies of $\one$ and the finite filtration splits in the semisimple ambient category, $H$ is also a direct sum of copies of $\one$. There is no positive homological degree in filtration below $p$, so every
generator of $E$ has a unique lift, and its reduced coproduct is zero for
filtration reasons.  The graded commutator of two such primitive lifts is
primitive and has filtration degree at most $2p-1$.  Its leading term, if nonzero,
would be a primitive element of $\wedge^\bullet E$ in even homological degree.  But
the primitive subspace of $\wedge^\bullet E$ is $E$, which is concentrated in odd
degrees.  Thus the lifts graded-commute and square to zero.  They define a
filtered Hopf map $\wedge^\bullet E\to H$ whose associated graded is an
isomorphism.
\end{proof}

\begin{proposition}\label{prop:Q-homology}
There is a canonical isomorphism of graded Hopf algebras
$
 H_\bullet(\cQ_\bullet(\g))\cong\wedge^\bullet E_\g,
$
and the left $U(\g)$-action on this homology is trivial.
\end{proposition}

\begin{proof}
The first assertion follows from Proposition~\ref{prop:degeneration} and
Lemma~\ref{lem:hopf-rigidity}. To prove the second assertion,  
define morphisms $h$ and $\mu_{\g,\cQ}$ by schematic formulas
\[
 h(x\otimes q):=x[1]q,
 \qquad \mu_{\g,\cQ}(x\otimes q):=xq
\]
for $x\in\g$ and homogeneous $q\in\cQ_n(\g)$.
Since $d$ is a derivation and $d(x[1])=x$, one has
\[
 d(x[1]q)+x[1]d(q)=xq.
\]
Thus $dh+h(1\otimes d)=\mu_{\g,\cQ}$, so multiplication by 
$\g$ is null-homotopic.
 Thus the action of $\g$ on homology is zero, and the $U(\g)$-action factors
through the augmentation.
\end{proof}

\subsection{End of proof of Theorem~\ref{thm:main-intro}}

Since $\Rep(\g)$ is semisimple, the functor $\uHom_{U(\g)}(-,V)$ is exact. 
Hence Proposition~\ref{prop:Q-homology} gives, naturally in $V$,
$$
 H^\bullet_{\CE}(\g,V)\cong
H^\bullet\uHom_{U(\g)}(\cQ_\bullet(\g),V)\cong 
\uHom_{U(\g)}(\wedge^\bullet E_\g,V)\cong(\wedge^\bullet E_\g)^*\otimes V^\g
 \cong\wedge^\bullet E_\g^*\otimes V^\g.
$$
The usual cup product on
$C^\bullet_{\CE}(\g)
 =\uHom_{U(\g)}(\cQ_\bullet(\g),\one)$
is the convolution product induced by the coproduct of the differential graded
Hopf algebra $\cQ_\bullet(\g)$.  Since all objects involved are finite and
$\Rep(\g)$ is semisimple, the canonical isomorphism
\[
 H^\bullet_{\CE}(\g)
 \cong \uHom_{U(\g)}(H_\bullet(\cQ_\bullet(\g)),\one)
 \cong H_\bullet(\cQ_\bullet(\g))^*
\]
identifies this cup product with the product dual to the coproduct on
$H_\bullet(\cQ_\bullet(\g))$.  We equip $H^\bullet_{\CE}(\g)$ with the
coproduct dual to multiplication on $H_\bullet(\cQ_\bullet(\g))$.  With this
convention the above identification is an isomorphism of graded Hopf algebras,
and Proposition~\ref{prop:Q-homology} gives
\[
 H^\bullet_{\CE}(\g)\cong\wedge^\bullet E_\g^*
\]
as graded Hopf algebras.
For arbitrary $V$, the action of scalar cochains on
$\uHom_{U(\g)}(\cQ_\bullet(\g),V)$ is the convolution action
$(a,f)\mapsto(a\otimes f)\circ\Delta$.  The exact-Hom identification on
homology carries this action to exterior multiplication on the first
factor of $\wedge^\bullet E_\g^*\otimes V^\g$.

Now, $C_\bullet^{\CE}(\g)=\wedge^\bullet\g$ is the
graded dual of $C^\bullet_{\CE}(\g)=\wedge^\bullet\g^*$.  Hence
\[
 H_\bullet^{\CE}(\g)\cong H^\bullet_{\CE}(\g)^*
 \cong\wedge^\bullet E_\g.
\]

We now identify these spaces with invariant tensors.  If
$\alpha\in(\wedge^n\g^*)^\g$, then alternating the identity
$x\cdot\alpha=0$ in $n+1$ variables gives
$2d_{\CE}\alpha=0$.  Since $p\neq2$, every invariant cochain is therefore a
cocycle.  The $\g$-actions on the Chevalley-Eilenberg cochain and chain
complexes are null-homotopic by the Cartan homotopy formulas, so they act
trivially on cohomology and homology.  Since $\Rep(\g)$ is semisimple, the
invariants functor is exact;
therefore
\[
 H^\bullet\bigl((\wedge^\bullet\g^*)^\g\bigr)
 \cong H^\bullet_{\CE}(\g)^\g
 =H^\bullet_{\CE}(\g).
\]
The differential on the complex on the left is zero, proving
$(\wedge^\bullet\g^*)^\g\cong H^\bullet_{\CE}(\g)$. Also the invariant
subspaces of $\wedge^n\g^*$ and $\wedge^n\g$ are in perfect duality (Subsection \ref{subsec:exterior-conventions}), hence the boundary map vanishes on
$(\wedge^\bullet\g)^\g$, and the same exact-invariants argument for the chain
complex gives
\[
 (\wedge^\bullet\g)^\g\cong H_\bullet^{\CE}(\g).
\]

If $V$ is simple and
$V^\g\neq0$, then $\g$ acts trivially on $V$; hence the cohomology of a simple
module with nonzero action vanishes.  The Poincar\'e polynomial follows from
the exterior-algebra formula.

It remains to identify the algebraically defined coproduct on $H^\bullet_{\CE}(\g)$ with the coproduct transported from de Rham cohomology in Proposition~\ref{prop:dR-CE}.
Let $G=\exp(\g)$.  All objects involved are finite.  Dualizing the
left-translation trivialization~\eqref{eq:left-trivialization-forms}, with the
Chevalley-Eilenberg differential~\eqref{eq:CE-cochain-differential}, gives an isomorphism of
differential graded Hopf algebras
\[
 \Omega^\bullet(G)^*
 \cong U(\g)\ltimes\wedge^\bullet \g
 =\cQ_\bullet(\g).
\]
Under this identification, wedge product of differential forms is dual to
the coproduct of $\cQ_\bullet(\g)$, and pullback by multiplication on $G$ is
dual to its product.  Thus the coproduct on $H^\bullet_{\CE}(\g)$ obtained
above is dual to multiplication on $H_\bullet(\cQ_\bullet(\g))$.

This completes the proof of Theorem~\ref{thm:main-intro}.

\begin{remark}\label{rem:derived-invariants}
Semisimplicity gives $\Ext^i_{\Rep(\g)}(\one,V)=0$ for $i>0$.  This does not
contradict Theorem~\ref{thm:main-intro}: the augmentation
$\cQ_\bullet(\g)\to\one$ is not a resolution, since its homology is
$\wedge^\bullet E_\g$.  This is precisely the almost-Koszul defect of
$\g$.
\end{remark}

\section{Examples}\label{sec:examples}

Examples of linearly reductive Lie algebras in $\Ver_p^+$
are discussed in \cite[Section~4]{CEO24}. Namely, 
let $Q$ be a connected Dynkin
diagram, let $G_Q^{\rm ad}$ be the adjoint group of type $Q$, and write
$\operatorname{Tilt}(G_Q^{\rm ad})$ for its category of tilting modules.  We
use the notation
\[
 \mathrm{SS}:\operatorname{Tilt}(G_Q^{\rm ad})
 \longrightarrow \Ver_p(G_Q^{\rm ad})
\]
for the quotient by negligible morphisms, where $\Ver_p(G_Q^{\rm ad})$ 
is the Verlinde category of $G_Q^{\rm ad}$ (\cite[Section~6]{EO22}), denoted
$\Ver_p^+(G_Q)$ in \cite[Section~4.1]{CEO24}.  This is the semisimplification
functor of \cite[Definition~2.5 and Theorem~2.6]{EO22}.

Let $Q$ have rank $r$ and exponents
$1=m_1\leq\cdots\leq m_r=h_Q-1$, where $h_Q$ is its Coxeter number.  For
$p>h_Q$, let $\g_{Q,\kk}$ be the simple Lie algebra over $\kk$ attached to
$Q$.  Under the forgetful functor
$\Ver_p(G_Q^{\rm ad})\to\Ver_p^+$, the object
$\mathrm{SS}(\g_{Q,\kk})$ is the Lie algebra $\g_{Q,p}$.
Let $F:\Ver_p(G_Q^{\rm ad})\to\Ver_p^+$ denote this forgetful functor.
By \cite[Remark~4.2(6)]{CEO24}, the group scheme reconstructed from $F$ is
$G_{Q,p}:=\exp(\g_{Q,p})$, and
\begin{equation}\label{eq:principal-reconstruction}
 \mathcal O(G_{Q,p})^*=U(\g_{Q,p}),\qquad
 \Ver_p(G_Q^{\rm ad})\simeq\Rep(G_{Q,p},\varphi),
\end{equation}
where $\varphi:\pi_1(\Ver_p^+)\to G_{Q,p}$ is the canonical
fundamental-group morphism.  In $\Rep(G_{Q,p},\varphi)$ the action through
$\varphi$ is required to be the canonical action.

To obtain semisimplicity for all internal $U(\g_{Q,p})$-modules, choose
representatives $S_i$ of the simple objects of $\Ver_p(G_Q^{\rm ad})$.
The coend construction gives
\begin{equation}\label{eq:reconstruction-matrix-algebras}
 \mathcal O(G_{Q,p})\cong\bigoplus_i F(S_i)^*\otimes F(S_i),
 \qquad
 U(\g_{Q,p})\cong\bigoplus_i\uHom(F(S_i),F(S_i))
\end{equation}
as coalgebras and algebras, respectively.  Each algebra summand in
\eqref{eq:reconstruction-matrix-algebras} is an internal matrix algebra,
whose module category is Morita equivalent to the semisimple category
$\Ver_p^+$.  Thus $\g_{Q,p}$ is linearly reductive in the sense of
Definition~\ref{def:linear-reductive}.
 
The simple summands of $\g_{Q,p}$ are obtained from the multiset $\{2m_i+1\}$ by deleting $p$
and deleting every $a>p$ together with $2p-a$ (folding); see
\cite[Subsection~4.2]{CEO24}.  If $d_1,\dots,d_s$ are the remaining labels,
then
$\g_{Q,p}\cong\bigoplus_{j=1}^sL_{d_j}$ as an object of $\Ver_p^+$; we call $s$ the {\it rank} of $\g_{Q,p}$. 
Consequently, Theorem~\ref{thm:main-intro} gives an isomorphism of graded Hopf algebras
 $$
 H^\bullet_{\CE}(\g_{Q,p})
 \cong\wedge^\bullet(\xi_1,\dots,\xi_s),
 \quad |\xi_j|=d_j.
$$
For example, if $Q=A_r$ and $p>r+1$, then
$$
 H^\bullet_{\CE}(\g_{A_r,p})
 \cong\wedge^\bullet(\xi_3,\xi_5,\dots,\xi_{2s+1}),
$$
where $s=\min(r,p-r-2)$.

If $p>2h_Q$, no deletion or folding occurs in the principal decomposition,
so
\begin{equation}\label{eq:stable-verlinde}
 \g_{Q,p}\cong\bigoplus_{i=1}^rL_{2m_i+1},
 \qquad
 H^\bullet_{\CE}(\g_{Q,p})
 \cong\wedge^\bullet(\xi_1,\dots,\xi_r),
 \quad |\xi_i|=2m_i+1.
\end{equation}

In particular, consider the examples in \cite[Subsection~4.3]{CEO24} 
with $s=2$. Such examples are of the form 
$\g=L_3\oplus L_d$, and hence
\[
 H^\bullet_{\CE}(\g)\cong\wedge^\bullet(\xi_3,\xi_d).
\]
The list is
\[
\begin{array}{lll}
\text{type} & \text{range} & H^\bullet_{\CE}(\g)\\[2pt]
A_2 & p\geq7 & \wedge^\bullet(\xi_3,\xi_5)\\
B_2=C_2 & p\geq11 & \wedge^\bullet(\xi_3,\xi_7)\\
D_2^* & p\geq11 & \wedge^\bullet(\xi_3,\xi_{p-2})\\
G_2 & p\geq17 & \wedge^\bullet(\xi_3,\xi_{11})\\
E_2^* & p=23 & \wedge^\bullet(\xi_3,\xi_{15})\\
E_2^{**} & p=37 & \wedge^\bullet(\xi_3,\xi_{23}).
\end{array}
\]
where $D_2^*=\g_{D_{\frac{p-1}{2}},p}$, $E_2^*=\g_{E_7,23}$, $E_2^{**}=\g_{E_8,37}$. 
In each case the Poincar\'e polynomial is $(1+t^3)(1+t^d)$.

\begin{remark}\label{rem:simple-reduction} Any linearly reductive Lie algebra $\g$ in $\Ver_p^+$ is semisimple, i.e., is a direct sum of simple Lie algebras.
Indeed, decompose the adjoint $\g$-module as
$\g=\bigoplus_i\g_i$ with $\g_i$ simple.  Each $\g_i$ is an ideal, and for
$i\ne j$ one has $[\g_i,\g_j]\subseteq\g_i\cap\g_j=0$.  Thus each $\g_i$ is
either abelian or simple as a Lie algebra.  The abelian case is impossible:
by Lemma~\ref{subal} such a summand would be linearly reductive, whereas, if
$I\subset S\g_i=U(\g_i)$ is the augmentation ideal, then
$U(\g_i)/I^2$ is a nonsplit extension of $\one$ by $\g_i$.

So the problem of computing the (co)homology of $\g$ reduces to the case of simple linearly reductive Lie algebras.
These Lie algebras have the canonical fundamental-group structure used in
\cite{CEO24,Ne}.  Indeed, let
$\mathfrak s:=\Lie\pi_1(\Ver_p^+)\cong L_3$; by
\cite[Remark~4.2(6)]{CEO24}, $\pi_1(\Ver_p^+)=\exp(\mathfrak s)$.
For a nonzero simple linearly reductive $\g$, its center vanishes and
Theorem~\ref{thm:main-intro} gives $H^1_{\CE}(\g,\g)=0$.
Thus the adjoint map
$\g\to\operatorname{Der}(\g)$ to the Lie algebra of derivations is an
isomorphism.  The canonical action of $\mathfrak s$ on $\g$ therefore
has a unique lift to a Lie algebra morphism $\mathfrak s\to\g$.
This morphism is injective, since $\mathfrak s$ is simple and its
canonical action on every nontrivial simple object of $\Ver_p^+$ is
nonzero.  Integrating it gives
$\pi_1(\Ver_p^+)\to\exp(\g)$ with the canonical conjugation action.
Hence $\g$ is based in the terminology of \cite{Ne}.

It was conjectured in \cite[Conjecture~4.1]{CEO24} that any such Lie algebra is of the form $\g_{Q,p}$
for $p>h_Q$. However, Joe Newton (\cite{Ne}, main theorem) recently established a counterexample
to this conjecture: the subobject $L_3\oplus L_{15}\oplus L_{27}$ in the linearly reductive Lie algebra $\g_{E_7,29}$ is a Lie subalgebra, hence is linearly reductive by Lemma \ref{subal}, but it does not appear in the list of \cite[Subsection~4.3]{CEO24} (it is obtained by semisimplifying
$E_8$ in characteristic $29$ with respect to a subregular unipotent).
At the moment, this is the only known counterexample.
The same theorem establishes the rank-$2$ classification conjectured in
\cite[Conjecture~4.4]{CEO24}, and the rank-$3$ classification after adding
this counterexample.
\end{remark} 

\section{Applications to ordinary Lie algebra cohomology}
\label{sec:ordinary-cohomology}

Theorem \ref{thm:main-intro} implies classical results on cohomology of ordinary semisimple Lie algebras. 

\begin{corollary}\label{cor:ordinary-naive-bound}
Assume that $\operatorname{char}(\kk)=p>\dim\g_{Q,\kk}$.  Then
\begin{equation}\label{eq:ordinary-CE-cohomology}
 H^\bullet_{\CE}(\g_{Q,\kk},\kk)
 \cong\wedge^\bullet(\eta_1,\dots,\eta_r),
 \qquad |\eta_i|=2m_i+1,
\end{equation}
as graded algebras. The same statement holds in characteristic zero. 
\end{corollary}

\begin{proof} First note that $\dim\g_{Q,\kk}\geq3h_Q-3$ in every simple type except $G_2$, where
$\dim\g_{Q,\kk}=14$ and $3h_Q-3=15$.  But a prime $p>14$ satisfies $p\geq17$, so in all cases 
$p>3h_Q-3$. 

Let $G=G_Q^{\rm ad}$ and $\mathfrak a=\g_{Q,\kk}$.  Write
$\widetilde{\mathfrak a}:=\mathrm{SS}(\mathfrak a)\in\Ver_p(G_Q^{\rm ad})$
and $\overline{\mathfrak a}:=F(\widetilde{\mathfrak a})=\g_{Q,p}$,
where $F:\Ver_p(G_Q^{\rm ad})\to\Ver_p^+$ is the forgetful functor.
Under~\eqref{eq:principal-reconstruction}, $\widetilde{\mathfrak a}$
corresponds to the adjoint $G_{Q,p}$-module $\g_{Q,p}$.
The highest weight of the adjoint
module $\mathfrak a$ is the highest root. Since $p>3h_Q-3$, this weight lies
in the bottom alcove, so $\mathfrak a$ and $\mathfrak a^*$ are simple tilting
modules.  Moreover, since $p>\dim\g_{Q,\kk}$, for every $0\leq n\leq \dim\g_{Q,\kk}$ the modules
$\wedge^n\mathfrak a$ and $\wedge^n\mathfrak a^*$ are images of
antisymmetrizing idempotents in their tensor powers, and hence are tilting.

Every weight of $\wedge^\bullet\mathfrak a$ is a sum of distinct roots and
zero weights.  If a dominant weight $\lambda$ occurs, then
$2\rho-\lambda$ is a sum of positive roots, so $\lambda\leq2\rho$.  Let
$\theta_s$ be the highest short root, so that $\theta_s^\vee$ is the highest
coroot and
$\langle\rho,\theta_s^\vee\rangle=h_Q-1$.  Thus
\[
 \langle\lambda+\rho,\theta_s^\vee\rangle
 \leq3h_Q-3<p.
\]
It follows that every indecomposable tilting summand of
$\wedge^n\mathfrak a$ and $\wedge^n\mathfrak a^*$ lies in the bottom alcove;
hence it is simple and non-negligible.  See
\cite[Chapter~II, Sections~5-6]{Jan03}.

The full additive subcategory generated by these simple bottom-alcove tilting
modules is semisimple, and the semisimplification functor is fully faithful
on it.  Every term and differential of the Chevalley-Eilenberg complex lies
in this subcategory.  Its kernels, images, and cohomology objects are again
direct sums of the same simple modules; hence semisimplification commutes
with the cohomology of this complex.  Moreover, for every $n\leq \dim\g_{Q,\kk}$,
\begin{equation}\label{eq:SS-exterior-powers}
 \mathrm{SS}(\wedge^n\mathfrak a^*)
 \cong\wedge^n\widetilde{\mathfrak a}^{\,*}
\end{equation}
in $\Ver_p(G_Q^{\rm ad})$.  The identifications~\eqref{eq:SS-exterior-powers}
carry the ordinary Chevalley-Eilenberg differential and multiplication to
those of $C^\bullet_{\CE}(\widetilde{\mathfrak a})$.  Hence there is an
isomorphism of algebra objects in $\Ver_p(G_Q^{\rm ad})$,
\begin{equation}\label{eq:SS-equivariant-cohomology}
 \mathrm{SS}\bigl(H^\bullet_{\CE}(\mathfrak a,\kk)\bigr)
 \cong H^\bullet_{\CE}(\widetilde{\mathfrak a}).
\end{equation}
In particular, \eqref{eq:SS-exterior-powers} and
\eqref{eq:SS-equivariant-cohomology} are $G_{Q,p}$-equivariant under
\eqref{eq:principal-reconstruction}.  Applying $F$ gives
\begin{equation}\label{eq:SS-underlying-cohomology}
 F\!\left(\mathrm{SS}\bigl(H^\bullet_{\CE}(\mathfrak a,\kk)\bigr)\right)
 \cong H^\bullet_{\CE}(\overline{\mathfrak a}).
\end{equation}
The inequality $p>\dim\g_{Q,\kk}$ also implies $p>2h_Q$.  Thus
\eqref{eq:stable-verlinde} and Theorem~\ref{thm:main-intro} identify the
right-hand side of~\eqref{eq:SS-underlying-cohomology} with the exterior
algebra on the $\xi_i$.  Its adjoint $G_{Q,p}$-action is trivial by the
Cartan homotopy and the height-one property.  It follows that
$H^\bullet_{\CE}(\widetilde{\mathfrak a})$ is a direct sum of copies of the
unit in $\Ver_p(G_Q^{\rm ad})$.  Full faithfulness on the bottom-alcove subcategory
then implies that $H^\bullet_{\CE}(\mathfrak a,\kk)$ is the corresponding
direct sum of trivial $G$-modules and that its multiplication is the exterior
multiplication.  This proves the result in characteristic $p$.
The characteristic zero case follows from the positive-characteristic result
for large primes by the standard spreading-out argument. 
\end{proof}

\subsection{Comparison with earlier results}

In characteristic zero, Corollary 
\ref{cor:ordinary-naive-bound} recovers the classical theorem of Borel and Chevalley.
In characteristic $p$, the Corollary holds in the larger range $p>3h_Q-3$ due to
Friedlander-Parshall \cite[Theorem~1.2]{FP86}. The
Frobenius-kernel, or restricted-cohomology, calculation used as an input in
their proof is due to Andersen-Jantzen
\cite[Corollary~3.7]{AJ84}; it does not by itself compute ordinary
Chevalley-Eilenberg cohomology.  Friedlander-Parshall combine this input
with the comparison spectral sequence between ordinary and restricted
cohomology:
\[
 E^2_{2i,j}=(S^i\g_{Q,\kk}^*)^{(1)}
 \otimes H^j_{\CE}(\g_{Q,\kk},\kk)
 \Longrightarrow H^{2i+j}_{\mathrm{res}}(\g_{Q,\kk},\kk),
\]
where $H^\bullet_{\mathrm{res}}$ denotes restricted Lie algebra cohomology; see \cite[p.~1079]{FP88} and \cite[p.~114]{Far91}.
In the range $p>3h_Q-3$, the Andersen-Jantzen calculation, in the form used
by Friedlander-Parshall, identifies the abutment as
\[
 H^{2\bullet}_{\mathrm{res}}(\g_{Q,\kk},\kk)^{(-1)}
 \cong\kk[\mathcal N],
 \qquad H^{2\bullet+1}_{\mathrm{res}}(\g_{Q,\kk},\kk)=0,
\]
where $(-1)$ denotes inverse Frobenius twist and $\mathcal N$ is the nilpotent cone of $\g_{Q,\kk}$.
Friedlander-Parshall identify the edge map with the quotient
\[
 S(\g_{Q,\kk}^*)\longrightarrow
 \kk[\mathcal N]
 \cong S(\g_{Q,\kk}^*)/(f_1,\dots,f_r),
 \qquad \deg f_i=m_i+1,
\]
where $f_1,\dots,f_r$ form a regular sequence.  They construct classes
$\eta_i\in H^{2m_i+1}_{\CE}(\g_{Q,\kk},\kk)$ whose successive
transgressions are
\[
 d_{2m_i+2}\eta_i=f_i^{(1)}
\]
modulo the preceding $f_j^{(1)}$.  These are precisely the differentials in
the Koszul complex of the regular sequence
$f_1^{(1)},\dots,f_r^{(1)}$.  The proof of
\cite[Theorem~1.2]{FP86} shows that these differentials exhaust the spectral
sequence; hence there are no additional classes or hidden multiplicative
extensions, and the ordinary cohomology is the exterior algebra in~\eqref{eq:ordinary-CE-cohomology}.

\section{Relative cohomology}\label{sec:relative-cohomology}

Let $\mathfrak k\subset\g$ be a Lie subalgebra.  By
Lemma~\ref{subal}, $\mathfrak k$ is linearly reductive.  Write
\[
 \mathfrak k_m:=\Hom_{\Ver_p^+}(L_m,\mathfrak k),
 \qquad
 E_{\mathfrak k}:=
 \bigoplus_{\substack{3\leq m\leq p-2\\m\text{ odd}}}
 \mathfrak k_m^{(1)}[m].
\]
The inclusion $\mathfrak k\hookrightarrow\g$
induces an injection $E_{\mathfrak k}\hookrightarrow E_\g$.  Set
\[
 E_{\g/\mathfrak k}:=E_\g/E_{\mathfrak k}
 \cong
 \bigoplus_{\substack{3\leq m\leq p-2\\m\text{ odd}}}
 (\g_m/\mathfrak k_m)^{(1)}[m].
\]
Thus $E_{\g/\mathfrak k}$ is the graded vector space $E_X$ associated
in Section~\ref{sec:proof-main} to the object $X=\g/\mathfrak k$.

For a finite $\g$-module $V$, define the relative Chevalley-Eilenberg
complex by
\[
 C^n_{\CE}(\g,\mathfrak k;V)
 :=\uHom_{U(\mathfrak k)}
 \bigl(\wedge^n(\g/\mathfrak k),V\bigr),
\]
where $\mathfrak k$ acts on $\g/\mathfrak k$ by the adjoint action and the
differential is the usual relative Chevalley-Eilenberg differential.  Its
cohomology is denoted $H^\bullet_{\CE}(\g,\mathfrak k;V)$, and for
$V=\one$ we omit $V$ from the notation.

\begin{theorem}\label{thm:relative-main}
Let $\g$ be a linearly reductive Lie algebra in $\Ver_p^+$ and let
$\mathfrak k\subset\g$ be a Lie subalgebra.  Then there is a canonical
isomorphism of graded algebras
\[
 H^\bullet_{\CE}(\g,\mathfrak k)
 \cong \wedge^\bullet E_{\g/\mathfrak k}^*.
\]
More generally, for every finite $\g$-module $V$ there is a natural
isomorphism
\[
 H^\bullet_{\CE}(\g,\mathfrak k;V)
 \cong \wedge^\bullet E_{\g/\mathfrak k}^*\otimes V^\g
\]
of graded modules over $H^\bullet_{\CE}(\g,\mathfrak k)$.  In particular,
the relative cohomology of a simple $\g$-module with nonzero $\g$-action
vanishes.

Under the canonical isomorphisms of Theorem~\ref{thm:main-intro}, the
natural maps
\[
 H^\bullet_{\CE}(\g,\mathfrak k)
 \longrightarrow H^\bullet_{\CE}(\g)
 \longrightarrow H^\bullet_{\CE}(\mathfrak k)
\]
identify with
\[
 \wedge^\bullet E_{\g/\mathfrak k}^*
 \lhook\joinrel\longrightarrow
 \wedge^\bullet E_\g^*
 \longrightarrow
 \wedge^\bullet E_{\mathfrak k}^*.
\]
Here the first arrow is induced by
$E_{\g/\mathfrak k}^*\hookrightarrow E_\g^*$, and the second by
$E_\g^*\twoheadrightarrow E_{\mathfrak k}^*$.  Consequently, restriction
$H^\bullet_{\CE}(\g)\to H^\bullet_{\CE}(\mathfrak k)$ is surjective, and
relative cohomology is the exterior subalgebra of $H^\bullet_{\CE}(\g)$
generated by the kernel of restriction on primitive elements.
\end{theorem}

\begin{proof}
\emph{Step 1: the relative complex and its PBW filtration.}
Set
\[
 A:=\cQ_\bullet(\g),\qquad
 B:=\cQ_\bullet(\mathfrak k),\qquad
 B^+:=\ker(B\xrightarrow{\varepsilon}\one),
\]
and define the relative quantized Koszul complex
\[
 \cQ_\bullet(\g,\mathfrak k)
 :=A\otimes_B\one=A/AB^+.
\]
Since $B$ is a differential graded Hopf subalgebra of $A$, the subobject
$AB^+$ is a differential graded left ideal, a right $B$-submodule, and a
coideal.  Hence
$\cQ_\bullet(\g,\mathfrak k)$ is a differential graded cocommutative
coalgebra and a left $U(\g)$-module.

There is a canonical morphism of left $U(\g)$-modules
\begin{equation}\label{morr}
 U(\g)\otimes_{U(\mathfrak k)}
 \wedge^n(\g/\mathfrak k)
 \longrightarrow \cQ_n(\g,\mathfrak k)
\end{equation}
which schematically sends
$u\otimes(\bar x_1\wedge\cdots\wedge\bar x_n)$ to the class of
$u x_1[1]\cdots x_n[1]$.  The relations in the smash product show that
this map is well defined. Moreover, it is an isomorphism.  Indeed, choose a
$\mathfrak k$-stable complement $\mathfrak m$ to $\mathfrak k$ in $\g$;
such a complement exists because $\Rep(\mathfrak k)$ is semisimple.
Equip $A$ and $B$ with the filtration of
Lemma~\ref{lem:quantized-koszul}.  PBW makes $B\hookrightarrow A$ strict
and gives
\begin{equation}\label{eq:relative-PBW-free-graded}
 \gr_F A\cong\cK_\bullet(\mathfrak m)\otimes\gr_F B
\end{equation}
as graded right $\gr_F B$-modules.  Split the finite filtrations in the
ambient semisimple category, and choose a filtered graded subobject
$T\subset A$ lifting $\cK_\bullet(\mathfrak m)\otimes\one$.
Multiplication $T\otimes B\to A$ has associated graded
\eqref{eq:relative-PBW-free-graded}, so it is a filtered isomorphism.
It identifies $T\otimes B^+$ with $AB^+$ and therefore gives
\begin{equation}\label{eq:relative-PBW-ideal}
 \gr_F(AB^+)=\gr_F(A)\gr_F(B)^+.
\end{equation}
The degree-zero part also identifies $U(\g)$ as a filtered free right
$U(\mathfrak k)$-module.  Thus, with exterior degree included in the
filtration, both sides of~\eqref{morr} have associated graded
$S\mathfrak m\otimes\wedge^n\mathfrak m$, and the associated graded of
\eqref{morr} is the identity.  Hence, canonically,
\begin{equation}\label{eq:relative-Q}
 \cQ_n(\g,\mathfrak k)
 \cong U(\g)\otimes_{U(\mathfrak k)}
 \wedge^n(\g/\mathfrak k).
\end{equation}
Under~\eqref{eq:relative-Q}, the differential is the standard relative
Chevalley-Eilenberg chain differential.  In particular,
\[
 \uHom_{U(\g)}\bigl(\cQ_\bullet(\g,\mathfrak k),V\bigr)
 \cong C^\bullet_{\CE}(\g,\mathfrak k;V).
\]

Give $\cQ_\bullet(\g,\mathfrak k)$ the quotient filtration induced by the
filtration of Lemma~\ref{lem:quantized-koszul}.  By
\eqref{eq:relative-PBW-ideal}, its associated graded differential graded
coalgebra is
\begin{equation}\label{eq:relative-associated-graded}
 \gr_F\cQ_\bullet(\g,\mathfrak k)
 \cong
 S(\g/\mathfrak k)\otimes\wedge^\bullet(\g/\mathfrak k)
 =\cK_\bullet(\g/\mathfrak k).
\end{equation}
\smallskip
\noindent\emph{Step 2: degeneration and the homology coalgebra.}
The corresponding spectral sequence has
\[
 E_2^{s,n}
 =H_n\bigl(\cK_\bullet(\g/\mathfrak k)\bigr)_s
 \cong \bigl(\wedge^{\bullet,\bullet}E_{\g/\mathfrak k}\bigr)_{s,n}
 \Longrightarrow H_n\bigl(\cQ_\bullet(\g,\mathfrak k)\bigr).
\]

The quotient map
\[
 q:\cQ_\bullet(\g)\longrightarrow
 \cQ_\bullet(\g,\mathfrak k)
\]
is filtered.  On the $E_2$-pages it induces the exterior-algebra map
\[
 \wedge^\bullet E_\g
 \longrightarrow \wedge^\bullet E_{\g/\mathfrak k}
\]
coming from the surjection $E_\g\twoheadrightarrow E_{\g/\mathfrak k}$;
this follows from the functoriality in
Proposition~\ref{prop:koszul-general}.  Hence the map of $E_2$-pages is
surjective.  The spectral sequence for $\cQ_\bullet(\g)$ degenerates at
$E_2$ by Proposition~\ref{prop:degeneration}.  It follows inductively that
the relative spectral sequence also degenerates: if the map on the
$r$-th pages is surjective, then every relative differential is the image
of an absolute differential and is therefore zero, after which
surjectivity persists to the $(r+1)$-th page.  Consequently,
\begin{equation}\label{eq:relative-gr-homology}
 \gr_F H_\bullet\bigl(\cQ_\bullet(\g,\mathfrak k)\bigr)
 \cong \wedge^{\bullet,\bullet}E_{\g/\mathfrak k},
\end{equation}
and the map induced by $q$ on homology is surjective.

The isomorphism of Proposition~\ref{prop:Q-homology} is functorial in the
Lie algebra: this follows from the naturality of the PBW spectral
sequence and the uniqueness of the primitive lifts in
Lemma~\ref{lem:hopf-rigidity}.  Therefore the map
\[
 H_\bullet(B)\longrightarrow H_\bullet(A)
\]
identifies with the natural inclusion
$\wedge^\bullet E_{\mathfrak k}\hookrightarrow
 \wedge^\bullet E_\g$.  Since $q$ is a morphism of right $B$-modules to a
module on which $B$ acts through the augmentation, the surjection on
homology factors through
\[
 H_\bullet(A)\otimes_{H_\bullet(B)}\one
 \cong
 \wedge^\bullet E_\g\otimes_{\wedge^\bullet E_{\mathfrak k}}\one
 \cong \wedge^\bullet(E_\g/E_{\mathfrak k}).
\]
The source and target of the resulting map have the same graded
dimension by~\eqref{eq:relative-gr-homology}.  Hence it is an isomorphism.
Since it is induced by the quotient morphism of differential graded
coalgebras, it is an isomorphism of graded coalgebras:
\begin{equation}\label{eq:relative-Q-homology}
 H_\bullet\bigl(\cQ_\bullet(\g,\mathfrak k)\bigr)
 \cong \wedge^\bullet E_{\g/\mathfrak k}.
\end{equation}

\smallskip
\noindent\emph{Step 3: coefficients and natural maps.}
As in Proposition~\ref{prop:Q-homology}, the left $U(\g)$-action on this
homology is trivial.  Namely, for $x\in\g$ the schematic formula
\[
 h_x(\overline q):=\overline{x[1]q}
\]
defines a homotopy on the quotient complex, and
\[
 d h_x+h_xd=\mu_x,
\]
where $\mu_x$ denotes left multiplication by $x$ on
$\cQ_\bullet(\g,\mathfrak k)$.

Finally, $\Rep(\g)$ is semisimple, so
$\uHom_{U(\g)}(-,V)$ is exact.  Using~\eqref{eq:relative-Q} and
\eqref{eq:relative-Q-homology}, we obtain
\[
\begin{aligned}
 H^\bullet_{\CE}(\g,\mathfrak k;V)
 &\cong
 \uHom_{U(\g)}\left(
 H_\bullet\bigl(\cQ_\bullet(\g,\mathfrak k)\bigr),V\right)\\
 &\cong
 \wedge^\bullet E_{\g/\mathfrak k}^*\otimes V^\g.
\end{aligned}
\]
For $V=\one$, the cup product is dual to the coproduct on
$H_\bullet(\cQ_\bullet(\g,\mathfrak k))$, so this is an algebra
isomorphism.  For general $V$, the module action is convolution
$(a,f)\mapsto(a\otimes f)\circ\Delta$.  Under the exact-Hom identification
on homology it becomes exterior multiplication on the first factor of
$\wedge^\bullet E_{\g/\mathfrak k}^*\otimes V^\g$.
The statements about the maps to absolute cohomology and
restriction follow by dualizing the coalgebra maps induced by
$\cQ_\bullet(\mathfrak k)\to\cQ_\bullet(\g)\to
 \cQ_\bullet(\g,\mathfrak k)$.
\end{proof}

\begin{corollary}\label{cor:relative-degrees}
Let $D_{\mathfrak a}$ denote the multiset in which each odd integer $m$
with $3\leq m\leq p-2$ occurs with multiplicity
$\dim\Hom_{\Ver_p^+}(L_m,\mathfrak a)$.  Then
$D_{\mathfrak k}$ is a submultiset of $D_\g$, and
$H^\bullet_{\CE}(\g,\mathfrak k)$ is the exterior algebra on one generator
for each occurrence in the multiset difference
$D_\g\mathbin{\setminus}D_{\mathfrak k}$.  Explicitly,
\[
\begin{gathered}
 H^\bullet_{\CE}(\g,\mathfrak k)
 \cong\wedge^\bullet(\xi_{m,j}),\qquad |\xi_{m,j}|=m,\\
 3\leq m\leq p-2,\quad m\text{ odd},\quad
 1\leq j\leq\dim\g_m-\dim\mathfrak k_m.
\end{gathered}
\]  In particular,
\[
 \sum_{n\geq0}\dim H^n_{\CE}(\g,\mathfrak k)t^n
 =\prod_{\substack{3\leq m\leq p-2\\m\text{ odd}}}
 (1+t^m)^{\dim\g_m-\dim\mathfrak k_m}.
\]
\end{corollary}

\begin{remark}
For the Lie algebras $\g_{Q,p}$ in the stable range $p>2h_Q$, the multiset
$D_\g$ is $\{2m_i+1\}$, where the $m_i$ are the exponents.  Thus
Corollary~\ref{cor:relative-degrees} gives precisely the exterior algebra
on the multiset difference $D_\g\setminus D_{\mathfrak k}$.  
This also recovers the classical relative-cohomology formula for an
inclusion $\mathfrak k_0\subset\mathfrak g_0$ of complex semisimple Lie
algebras containing a nilpotent element that is principal in both.
Let $D^{\rm cl}_{\mathfrak a}$ be the multiset
$\{2m_i(\mathfrak a)+1\}$, where $m_i(\mathfrak a)$ are the exponents of
$\mathfrak a$.  Write the multiset difference
$D^{\rm cl}_{\mathfrak g_0}\setminus D^{\rm cl}_{\mathfrak k_0}$ as
$\{d_1,\ldots,d_s\}$.  Then
\begin{equation}\label{eq:classical-relative-common-principal}
 H^\bullet_{\CE}(\mathfrak g_0,\mathfrak k_0;\mathbb C)
 \cong\wedge^\bullet(\eta_1,\ldots,\eta_s),\qquad |\eta_i|=d_i.
\end{equation}
The same formula holds for the corresponding reductions in sufficiently
large positive characteristic.

Indeed, complete the common nilpotent to an $\mathfrak{sl}_2$-triple in
$\mathfrak k_0$; this triple is principal in $\mathfrak g_0$ as well.
Thus the principal decompositions make
$D^{\rm cl}_{\mathfrak k_0}$ a submultiset of
$D^{\rm cl}_{\mathfrak g_0}$.
Choose a model of the inclusion and the triple over a finitely generated
$\mathbb Z$-algebra.  After localization and in sufficiently large
residue characteristic, all modules occurring in the relative cochain
terms and in the maps defining $\mathfrak k_0$-invariants are direct sums
of bottom-alcove simple $\mathrm{SL}_2$-modules.  Semisimplification along
this common principal $\mathrm{SL}_2$ therefore preserves the split
inclusion and commutes with forming relative cochains and their
cohomology.  The relative cochains are trivial $\mathrm{SL}_2$-modules,
since the triple lies in $\mathfrak k_0$; in this bottom-alcove range,
vanishing of the $\mathfrak{sl}_2$-action implies triviality of the
$\mathrm{SL}_2$-action.  Theorem~\ref{thm:relative-main} now gives
\eqref{eq:classical-relative-common-principal} in these characteristics.
The characteristic-zero statement follows by spreading out, as in
Section~\ref{sec:ordinary-cohomology}.
\end{remark}

\end{document}